\documentclass[a4paper,10pt]{article}
\usepackage[T1]{fontenc}
\usepackage[utf8]{inputenc}
\usepackage[english]{babel}
\usepackage{amsmath}
\usepackage{amsthm}
\usepackage{amsfonts}
\usepackage{amssymb}
\usepackage{mathrsfs}
\IfFileExists{srcltx.sty}{\usepackage{srcltx}}

\title{Effective Hasse principle for the intersection of two quadrics}
\author{Tony Quertier \\ Université de Caen Normandie, France \\ tony.quertier@unicaen.fr}

\theoremstyle{plain} \newtheorem{theorem}{Theorem}
\theoremstyle{plain} \newtheorem{algorithm}{Algorithm}
\theoremstyle{plain} \newtheorem{proposition}[theorem]{Proposition}
\theoremstyle{plain} \newtheorem{corollary}[theorem]{Corollary}
\theoremstyle{plain} \newtheorem{lemma}[theorem]{Lemma}
\theoremstyle{definition} \newtheorem{definition}{Definition}

\theoremstyle{remark} \newtheorem{remark}{Remark}

\theoremstyle{remark} \newtheorem{notation}{Notation}

\theoremstyle{definition} \newtheorem{hypothesis}{Hypothesis}

\def\QQ{\mathbb{Q}}
\def\CC{\mathbb{C}}

\def\RR{\mathbb{R}}
\def\ZZ{\mathbb{Z}}
\def\PP{\mathbb{P}}
\def\k*{k^{\ast}}
\def\q*{Q^{\ast}}

\DeclareMathOperator{\Id}{Id}

\DeclareMathOperator{\im}{Im}
\DeclareMathOperator{\re}{Re}

\def\GLn{GL_n}

\def\GL2{GL_2}

\DeclareMathOperator{\trace}{trace}

\begin{document}

 \maketitle
 
 \begin{abstract}
We consider a smooth system of two homogeneous quadratic equations over $\QQ$ in $n \geq 13$ variables. In this case, the Hasse principle
is known to hold, thanks to the work of Mordell in 1959. The only local obstruction is over $\RR$. In this paper, we give an explicit algorithm to decide whether a
nonzero rational solution exists, and if so, compute one.
 \end{abstract}

\section{Introduction} 

Let $F_1,\ldots, F_m$ be polynomials in the variables $x_1, \ldots, x_n$ with coefficients in a number field $K$. In the study of the 
resolution of the system $F_1(x_1 , \ldots ,x_n)= \ldots =F_m(x_1,\ldots,x_n)=0$, three very natural and well-studied problems are:

\begin{itemize}
 \item LP (= Local problems): Decide whether solutions exist in every completion of $K$ ($\RR$, $\CC$, $p$-adic fields, \ldots); 
   if so, compute them.

 \item HP (= Hasse Principle): If LP is true, show the existence of a solution  in $K$. Otherwise, there are clearly no 
  solutions in $K$.

 \item EHP (= Effective HP): If solutions exist in $K$, compute one.
\end{itemize}

In this paper, we consider a smooth system of two homogeneous quadratic equations over $K=\QQ$ in $n \geq 13$ variables. Before studying the case of
two equations, it is worth recalling what is known in the case of a single quadratic equation.

Let $q(x_1,\ldots,x_n)$ be a homogeneous quadratic form 
over $\QQ$ and $q(x_1,\ldots,x_n)=0$ the associated quadratic equation. 
The LP question was solved by the Chevalley-Warning theorem and by Hensel's lemma \cite{Ser}. The corresponding algorithms are 
particularly efficient. The Hasse-Minkowski theorem asserts that HP holds for a single quadratic equation. 

To solve EHP, Simon \cite{DS} and Castel \cite{PC} have written algorithms that quickly compute an explicit rational solution of 
$q(x_1,\ldots,x_n)=0$.
Consequently, for a single quadratic equation, we consider the three problems solved and now focus on the case of
two quadratic equations.

Let $q_0(x_1,\ldots,x_n)$, $q_1(x_1,\ldots,x_n)$ be two quadratic forms over $\QQ$.
Demyanov \cite{Demy} and Birch, Lewis and Murphy \cite{BLM} solved LP for $n \geq 9$. 
Many people have worked on the HP problem for two quadratic forms. Let us mention the most general results.
Mordell settled the case $n \geq 13$ in 1959 \cite{Mor59}. His result was lowered down to $n \geq 11$ by Swinnerton-Dyer
in 1964 \cite{SwD64}, and later to $n \geq 9$ by Colliot-Thélène, Swinnerton-Dyer and Sansuc
in 1987 \cite{CTSWS}. In 2006, Wittenberg
 \cite{OW2006} proved that, if we assume Schinzel's hypothesis and the finiteness of Tate-Shafarevich
groups of elliptic curves over number fields, then the HP holds as soon as $n \geq 6$.

To our knowledge, no work exists on the EHP problem for two quadratic equations.
In this paper, we give explicit algorithms to solve EHP for $n \geq 13$. A non-negligible part of our work is based on 
\cite{Mor59}. 

\bigskip 

In a preliminary part, we fix the notation and recall the notion of smoothness.

In part \ref{section:realqf}, we study the different signatures of the forms in the pencil, which govern the existence of a real solution. 
For this, we introduce the simultaneous block diagonalization of two quadratic forms and show the existence of a quadratic form
with a convenient signature. This leads to a simple algorithm that decides the existence of a real solution.

In part \ref{section:reduction}, we give some low-level algorithms to split off a quadratic form into hyperbolic planes 
over $\RR$ or $\QQ$. These rely on the ability to compute a solution for a single quadratic equation. Over $\QQ$, as already 
mentioned, we may use the algorithm of Castel \cite{PC}.

Part \ref{section:realsolution} is devoted to the computation of an explicit nontrivial real solution of the system.

In part \ref{section:changebasis}, using this real solution, we can construct a rational totally isotropic subspace for $q_0(x)$ such 
that $q_1(x)$ is indefinite over this subspace. 

In the last part \ref{section:rationalsolution}, we use this subspace to derive a nontrivial rational solution of the system.

\section{General notation}

Let $K \supset \QQ$ be a  field. Let $q_0$ and $q_1$ be two quadratic forms over $K$ in $n$ variables. 
Using the canonical basis of $K^n$ we have:
 
 \[
  q_0(x)=\sum_{i,j=1}^na_{ij}x_ix_j, \ q_1(x)=\sum_{i,j=1}^nb_{ij}x_ix_j
 \]
with $a_{ij}=a_{ji}$ and $b_{ij}=b_{ji}$. We write $Q_0=(a_{ij})$, $Q_1=(b_{ij})$ the associated symmetric matrices. 
For $x=(x_1,\ldots,x_n)$, we have \[ q_0(x)=xQ_0{}^tx, \ q_1(x)=xQ_1{}^tx. \]
We also use the notation $q_0(x,y)= x Q_0 {}^ty$ for the associated bilinear form.

Let 
 \[
 V=\{ x \in \PP^{n-1}\left( \overline{K}\right) \ |  \ q_0(x)=q_1(x)=0 \}
\]
be the projective variety defined by the two quadrics associated to $q_0$ and $q_1$.

To study the intersection of two quadrics, it is necessary to
study the pencil of quadrics through $V$. We denote by $\mathcal{P}_K$ the pencil of quadrics associated to the pair $(q_0,q_1)$, 
that is the family of quadrics $a_0 q_0 + a_1 q_1=0$ with $(a_0:a_1) \in \PP^1(K)$. In practice, we will mainly consider this pencil for
$K=\QQ$ and $K=\RR$. If $\det(Q_0)=0$, we replace $q_0$ by $a_0 q_0 + a_1 q_1$ for some 
$(a_0:a_1) \neq (0:1)$
to assure that $\det(Q_0) \neq 0$ and similarly for $Q_1$. From now on, we can set $\lambda= \frac{a_0}{a_1}$ and 
$\Delta(\lambda)= \det(\lambda Q_0+Q_1)$: this is a polynomial of degree exactly $n$ in $\lambda$.

 We denote by $[r,s]$ the \emph{signature} of a quadratic form $q$ in $n$ variables, where $r$ is the 
 positive component and $s$ the negative component. 
 If $\det(Q) \neq 0$, we have $r+s=n$, otherwise we have $r+s < n$.
 
  We denote by $\oplus$ the traditional \emph{orthogonal sum} for quadratic forms. Moreover, for two matrices $A$ and $B$, we define
 $A \oplus B$ the block diagonal matrix $ \begin{pmatrix} A & 0 \\ 0 & B\end{pmatrix}$.

The variety $V$ is \emph{smooth} if $q_0=0$ and $q_1=0$ intersect transversally \textit{i.e}. if the rank of the Jacobian of $q_0$ and $q_1$ 
is equal to $2$ at every point of $V$. 

 \begin{hypothesis}
  We say that two quadratic forms $q_0$ and $q_1$(resp. two symmetric matrices $Q_0$ and $Q_1$) defined over $K$, satisfy the \emph{hypothesis H} if $\det(Q_0) \neq 0$, $\det(Q_1) \neq 0$, and $V$ is
  smooth over $K$.
 \end{hypothesis}

We know, for example from \cite{Harris}, that the variety $V$ is smooth if and only if 
 $\Delta(\lambda)=\det( \lambda Q_0 + Q_1)$ has only simple roots in $\overline{K}$. We have therefore the equivalent formulation:

 \begin{hypothesis}
  We say that two quadratic forms $q_0$ and $q_1$ (resp. two symmetric matrices $Q_0$ and $Q_1$) defined over $K$, satisfy the \emph{hypothesis H} if $\det(Q_0) \neq 0$, $\det(Q_1) \neq 0$, and 
  $\Delta(\lambda)=\det( \lambda Q_0 + Q_1)$ has only simple roots in $\overline{K}$ .
 \end{hypothesis}

\section{Real quadratic forms} \label{section:realqf}

\subsection{Simultaneous diagonalization}

\begin{proposition}\label{Diag}
 Let $Q_0$ and $Q_1$ be two matrices of size $n$ satisfying Hypothesis H over $\RR$.
 Let $m$ be the number of real roots of $\Delta(\lambda)$. There exists a matrix $P \in \GLn(\RR)$ such that $PQ_0{}^tP$ is diagonal, with only $\pm 1$ on the diagonal, 
 and $PQ_1{}^tP$ is a block diagonal matrix, with $m$ first blocks of size $1$ and then $(n-m)/2$ blocks of size $2$ of the form
\[\begin{pmatrix}
a & b \\
b & -a
\end{pmatrix}. \] Furthermore, each such block in $PQ_1{}^tP$ is face to face with a block $\begin{pmatrix} 1 & 0 \\ 0 & -1 \end{pmatrix}$ in $PQ_0{}^tP$.
Algorithm \ref{diagonalisation} computes such a matrix $P$.
\end{proposition}

\begin{algorithm}\label{diagonalisation}
 Let $Q_0$ and $Q_1$ be two matrices of size $n$ satisfying Hypothesis H over $\RR$.
 This algorithm computes a matrix $P \in \GLn(\RR)$ satisfying the conclusion of Proposition \ref{Diag}.
 \end{algorithm}

\begin{ttfamily} 
 \begin{enumerate}
  \item Let $\Delta(\lambda)= \det(\lambda Q_0 + Q_1)$ and $\Lambda=\{\lambda_1,\ldots,\lambda_n\}$ be the list of the roots of $\Delta(\lambda)$ such that 
        $\lambda_1,\ldots,\lambda_m \in \RR$ and $\lambda_{m+i}=\overline{\lambda_{m+i+1}}$ for $i \geq 1$ odd.
  \item For $i$ from $1$ to $n$, find a generator $v_i$ of $\ker ( \lambda_i Q_0 + Q_1)$ such that, for $i \leq m$, $v_i \in \RR^n$.
  \item Set $j=m+1$. While $j < n$, set $w_j=\re(v_j)$, $w_{j+1}= \im(v_j)$ and $j=j+2$.
  \item For $i$ from $1$ to $m$, set $w_i=v_i$. 
  \item Let $P$ be the square matrix of size $n$ whose the $i$-th line is $w_i$, for $1 \leq i \leq n$. Set $Q'_0=PQ_0{}^tP$. 
  \item For $i$ odd from $1$ to $n-m-1$, apply the Jacobi's eigenvalues algorithm to diagonalize the block $(Q'_{0_{k,l}})_{m+i \leq k,l \leq m+i+1} $ and denote by $P_{m+i}'$ 
        the associated transformation matrix of size $2$.
  \item Set $P'=\Id(m) \oplus \displaystyle\bigoplus_i P_{m+i}'$. Set $Q''_0=P'Q'_0{}^tP'$ and $P=P'P$.
  \item For $i$ from $1$ to $n$, divide the $i$-th line of $P$ by $\sqrt{|Q_{0_{ii}}''|}$.
  \item Return $P$.
 \end{enumerate}
 \end{ttfamily}

\begin{proof}
In Step $2$, the dimension of each kernel is $1$ because $\Delta(\lambda)$ has only simple roots.
 We know that \[v_i(\lambda_iQ_0) + v_iQ_1=0\] then we have:
 \[ \begin{array}{rcl}
    q_1(v_i,v_j)&=&v_iQ_1{}^tv_j \\
     &=& v_i(-\lambda_iQ_0){}^tv_j  \\
     &=& -\lambda_i v_iQ_0{}^tv_j\\
     &=& -\lambda_i q_0(v_i,v_j)\\
    \end{array} \]
but we also have:
\[ \begin{array}{rcl}
    q_1(v_i,v_j)=-\lambda_j q_0(v_i,v_j).  \\
    \end{array} \] 
We have $\lambda_i \neq \lambda_j$ for $i \neq j$,  so we deduce that $q_0(v_i,v_j)=q_1(v_i,v_j)=0$. Since the $v_i$ are orthogonal for $q_0$ and $q_1$, the $w_i$ are also pairwise orthogonal for $q_0$ and $q_1$, except maybe
$w_{m+i}$ and $w_{m+i+1}$, for $i$ odd.
The matrices $Q_0'$ and $Q_1'$ are therefore block diagonal with blocks of size $1$ for each real root $\lambda$ and of size $2$ for each pair of conjugate complex roots.
For $i$ odd, from  the equality $q_0(v_{m+i},v_{m+i+1})=0$
  we deduce that $q_0(w_{m+i},w_{m+i})=-q_0(w_{m+i+1},w_{m+i+1})$. 
 So, the shape of the blocks of size $2$ associated to conjugate complex roots is:
 \[\begin{pmatrix}
a & b \\
b & -a
\end{pmatrix}. \]
The same is true for $Q_1'$.
In Step $8$, The Jacobi's eigenvalues algorithm computes an orthogonal matrix $P_{m+i}'$ to diagonalize the blocks $A_i=(Q'_{0_{k,l}})_{m+i \leq k,l \leq m+i+1}$. We consider $B_i=(Q'_{1_{k,l}})_{m+i \leq k,l \leq m+i+1}$. 
Since, $P'_{m+i}$ is orthogonal, we have:
\[\trace(P'_{m+i}A_i{}^tP'_{m+i})=\trace(A_iP'_{m+i}{}^tP'_{m+i})=\trace(A_i)=0\]  and then the shape of the blocks $P'_{m+i}A_i{}^tP'_{m+i}$ is always:
  \[\begin{pmatrix}
a & 0 \\
0 & -a
\end{pmatrix}. \]
Similarly, the shape of the $P'_{m+i}B_i{}^tP'_{m+i}$ is
  \[\begin{pmatrix}
c & d \\
d & -c
\end{pmatrix}. \]
At the level of blocks, Step 8 divides the two lines of $P'_{m+i}$ by the same constant $|a|$ therefore, the trace of the blocks is always zero.

\end{proof}

\subsection{Existence of a balanced quadratic form}

\begin{definition}
 We say that a quadratic form with signature $[r,s]$ is \emph{balanced} if $|r-s| \leq 1$.
\end{definition}

In this part we want to determine if a pair of quadratic forms has non trivial real solutions. After this we study the existence of a balanced quadratic form
in the pencil $\mathcal{P_\RR}$ and compute one if it exists.

\begin{lemma}[(Cauchy's bound)]
Let $P(x)= x^n + a_{n-1} x^{n-1} \cdots + a_0$ be a monic polynomial of degree $n$. If $x \in \CC$ is a root of $P$ then
$|x| \leq 1 + \max_{1 \leqslant i \leqslant n}(| a_i |)$. The constant $a=1 + \max_{1 \leqslant i \leqslant n}(| a_i |)$ is called
\emph{Cauchy's bound} of $P$. 
 \end{lemma}

The next result follows easily from Proposition \ref{Diag}. We use the notation $r(\lambda_i^+)$ for $\lim_{\lambda \to \lambda_i, \lambda > \lambda_i} r(\lambda)$ and similarly for $r(\lambda_i^-)$ with $\lambda< \lambda_i$.

\begin{theorem}\label{2}
Let $Q_0$ and $Q_1$ be two matrices of size $n$ satisfying Hypothesis H over $\RR$. We denote by $\lambda_1 < \ldots < \lambda_m$ the real roots
of $\det(\lambda Q_0+Q_1)$, $Q_{\infty}=Q_0$ and $Q_{-\infty}=-Q_0$. We also denote by $[r(\lambda),s(\lambda)]$ the signature of 
$\lambda Q_0+Q_1$ with $\lambda \in [-\infty,\infty]$. Then :
 \begin{enumerate}
  \item The signature of $\lambda Q_0+Q_1$ is constant over the intervals $[-\infty,\lambda_1[$, $]\lambda_i,\lambda_{i+1}[$ and $]\lambda_m,\infty]$ for $i \in \{1,\ldots,m-1\}$. 
  \item $r(\lambda_i^+)-r(\lambda_i^-)=-(s(\lambda_i^+)-s(\lambda_i^-))$.
  \item $r(\lambda_i)=\frac{1}{2}(r(\lambda_i^+)+r(\lambda_i^-))$ and $s(\lambda_i)=\frac{1}{2}(s(\lambda_i^+)+s(\lambda_i^-))$.
 \end{enumerate}

\end{theorem}

\begin{corollary}\label{equilibre}
 There exists $\lambda \in \QQ$ such that $\lambda q_0 + q_1$ is \emph{balanced}.
\end{corollary}

\begin{definition}
 We define the function $d : \RR \rightarrow \ZZ$ by 
 $d(\lambda)=r-s$, where  $[r,s]$ is the signature of $\lambda q_0 + q_1$.
\end{definition}

We can reformulate Theorem \ref{2} using the function $d$.

\begin{corollary}\label{fonction}
 The function $d$ is piecewise constant with discontinuities at the roots of $\Delta(\lambda)$. The value of $d$ at a 
 discontinuity is the average of the two limit values of $d$ on the left and on the right of this discontinuity.
\end{corollary}

It is convenient for the next lemma to use the notation $\lambda_0=-\infty$ and $\lambda_{n+1}=+\infty$.

\begin{lemma}\label{4}
 Let $Q_0$ and $Q_1$ be two matrices of size $n$ satisfying Hypothesis H over $\RR$.  Assume that $\Delta(\lambda)$ has $m \leq n$ real roots denoted by $\lambda_1 < \ldots < \lambda_m$. We have:
 \begin{enumerate}
  \item If $m \neq n$ then $\lambda Q_0 + Q_1$ is never definite.
  \item If $m=n$, there exists at most one interval $]\lambda_i,\lambda_{i+1}[$ over which $\lambda Q_0+Q_1$ is positive (resp. negative) definite. Moreover this interval is $]\lambda_s,\lambda_{s+1}[$ 
        (resp. $]\lambda_r,\lambda_{r+1}[$) where $[r,s]$ is the signature of $Q_0$.
 \end{enumerate}

\end{lemma}

% \begin{proof}
%  Assume that $d(-\infty)=a$, then $d(\infty)=-a$. If there exists a real $\lambda$ such that $\lambda Q_0+Q_1$ is positive (resp. negative) definite then we have $d(\lambda)=n$ 
%  (resp. $d(\lambda)=-n$). Because $d$ changes by $\pm 2$ through each $\lambda_i$, going from $a$ to $n$ requires at least $(n-a)/2$ steps (resp. $(n+a)/2$) and going from $n$ to $-a$ requires at least $(n+a)/2$ 
%  steps (resp. $(n-a)/2$).
%  Since $(n-a)/2+(n+a)/2=n$ we need exactly $n$ real roots, otherwise $\lambda Q_0 + Q_1$ cannot be definite.
% The observation that $s=(n-a)/2$ (resp. $r=(n+a)/2$) gives the conclusion.
% \end{proof}

\begin{theorem}\label{test}
 Let $q_0$ and $q_1$ be two quadratic forms of $n$ variables. Then $V(\RR) \neq \emptyset$ if and only if all the forms in $\mathcal{P_\RR}$ are
 indefinite.
\end{theorem}

Two different proofs are done in \cite{Mor59} and \cite{SwD64}. 
 \begin{algorithm}\label{signa}
 Let $Q_0$ and $Q_1$ be two matrices of size $n$ satisfying Hypothesis H over $\RR$. This algorithm computes a rational number $\lambda$ 
 such that $\lambda Q_0 + Q_1$ is definite if there exists one, and returns a message otherwise.
 \end{algorithm}

\begin{ttfamily} 
 \begin{enumerate}
  \item Let $a$ be Cauchy's bound of $\Delta(\lambda)=\det (\lambda Q_0 + Q_1)$ and $I$ the interval $[-a-1,a+1]$.
  \item Set $m$ the number of real roots of $\Delta$. If $m \neq n$, return a message saying that $\lambda Q_0+Q_1$ is never definite.
  \item Denote by $\lambda_1< \ldots < \lambda_n$ the roots of $\Delta(\lambda)$ and $[r,s]$ the signature of $-a Q_0+Q_1$.
  \item Let $\lambda$ and $\mu$ be two rational numbers such that $\lambda \in ] \lambda_r , \lambda_{r+1}[$ and $\mu \in ] \lambda_s, \lambda_{s+1} [$. 
  \item If $\lambda Q_0 + Q_1$ is definite, return $\lambda$. 
  \item If $\mu Q_0 + Q_1$ is definite, return $\mu$.
  \item Return a message saying that $\lambda Q_0+Q_1$ is never definite.
  
 \end{enumerate}
 \end{ttfamily}

This algorithm is an effective test of Theorem \ref{test}. We are able to decide whether $V(\RR) \neq \emptyset$ or equivalently  $q_0$ and $q_1$ have a common nonzero real solution using this 
algorithm. The explicit construction of a real solution will be done in Algorithm \ref{Qf13reel} when $n \geq 3$.

 \begin{algorithm}\label{1}
 Let $Q_0$ and $Q_1$ be two matrices of size $n$ satisfying Hypothesis H over $\RR$. This algorithm computes
 a rational number $\lambda$ such that $\lambda Q_0+Q_1$ is balanced and $\det(\lambda Q_0 + Q_1) \neq 0$.
 \end{algorithm}

 \begin{ttfamily} 
 \begin{enumerate}
  \item Let $a$ be Cauchy's bound of $\Delta(\lambda)=\det(\lambda Q_0+Q_1)$. 
  \item Set $\lambda_{max}=a+1$, $\lambda_{min}=-\lambda_{max}$, and $\lambda_b=0$.
  \item If $|d(\lambda_b)| \leq 1$ and $\Delta(\lambda_b) \neq 0$, return $\lambda_b$.
  \item If $d(\lambda_b)=0$, set $\lambda_{max}=\lambda_b$ and go to Step $7$.
  \item If $d(\lambda_b)$ and $d(\lambda_{min})$ have opposite signs, set $\lambda_{max}=\lambda_b$, else  set $\lambda_{min}=\lambda_b$.
  \item Set $\lambda_b= (\lambda_{min}+\lambda_{max})/2$ and go to Step $3$.
  
 \end{enumerate}
 \end{ttfamily}

\begin{proof}
Algorithm \ref{1} simply makes a dichotomy over $[\lambda_{min},\lambda_{max}]$ for function $d$, using the fact that $d(\lambda_{max})=-d(\lambda_{min})$. 
\end{proof}

 \section{Reduction of a balanced quadratic form}\label{section:reduction}

\begin{notation} 
We set $K$ a field, with $K=\RR$ or $K=\QQ$. We denote by
\[ \mathbb{H}=\begin{pmatrix}
0 & 1 \\
1 & 0
\end{pmatrix} \] the matrix associated to the quadratic form $2xy$: we call it a \emph{hyperbolic plane}.
\end{notation}

Now, we are going to give a set of algorithms to compute some transition matrix $P$ such that
$PQ_0{}^tP$ is the block diagonal matrix  $\mathbb{H} \oplus Q_2$.
In this section, we only consider indefinite quadratic forms over $K$ of dimension $n \geq 5$. 

\begin{notation}
 In this section, most algorithms take as an input a matrix $Q_0=(a_{ij})$ associated to a quadratic form $\sum_{i,j=1}^n a_{ij}x_ix_j$ defined over $K$ and compute a matrix $P \in \GLn(K)$. We denote by $a_{ij}'$ the entries of $PQ_0{}^tP$.
\end{notation}

\begin{algorithm}\label{BaseWitt}
 Let $Q_0=(a_{ij})$ be such that $\det(Q_0) \neq 0$ and $y$ a nonzero vector in $K^n$ such that $yQ_0{}^ty=0$. 
 This algorithm computes a matrix $P \in \GLn(K)$ such that $a_{11}'=0$ and $a_{12}' \neq 0$.
\end{algorithm}

\begin{ttfamily}
 \begin{enumerate}
  \item Let $P$ be a square matrix of size $n$ having $y$ as first line. \newline Complete the matrix $P$ to have $P \in \GLn(K)$.
  \item Set $Q_0'=PQ_0{}^tP$. Let $i \geq 2$ be the smallest index such that $(Q_0')_{1i} \neq 0$, and $P'$ be the transposition matrix which exchanges 
        the second line with the $i$-th line. 
  \item Return $P=P'P$.
  \end{enumerate}
\end{ttfamily}

\begin{algorithm}\label{Mordell1}
 Let $Q_0=(a_{ij})$ be such that $a_{11}=0$ and $a_{12} \neq 0$. 
 This algorithm computes a matrix $P \in \GLn(K)$ such that $a'_{11}=0$, $a'_{12}=1$, and $a'_{1i}=0$ for $3 \leq i \leq n$.
 \end{algorithm}

\begin{ttfamily} 
 \begin{enumerate}
  \item Set $P=\Id(n)$ and divide the first line of $P$ by $a_{12}$.
  \item Set $Q_0'=PQ_0{}^tP$.
  \item $P'=\Id(n)$. For $i=3$ to $n$, set $P'_{i2}=-(Q_0')_{i1}$.
  \item Return $P=P'P$.
  
 \end{enumerate}
 \end{ttfamily}
 
 \begin{algorithm}\label{Mordell2}
 Let $Q_0=(a_{ij})$ be such that $a_{11}=0$, $a_{12}=1$, and $a_{1i}=0$ for $3 \leq i \leq n$. This algorithm computes
 $P \in \GLn(K)$ such that $PQ_0{}^tP$ is of the form $\mathbb{H} \oplus Q_2$.
 \end{algorithm}
 
 \begin{ttfamily}
 
 \begin{enumerate}
  \item Set $P=\Id(n)$ and $S=\Id(n)$.
  \item For $i=2$ to $n$, set $P_{i1}=-a_{i2}$. Set $Q_3=PQ_0{}^tP$.
  \item For $i=1$ to $n$, set $S_{2i}=2S_{2i}-(Q_3)_{22}S_{1i}$.
  \item Set $S_{11}=1/2$ and $P=SP$.
  \item Return $P$.
 \end{enumerate}
 \end{ttfamily}

\begin{algorithm}\label{Redqf}
 Let $Q_0=(a_{ij})$ and $y$ be a nonzero vector in $K^n$ such that $yQ_0{}^ty=0$. This algorithm computes
 a matrix $P \in \GLn(K)$ such that $PQ_0{}^tP$ is of the form $\mathbb{H} \oplus Q_2$.
 \end{algorithm}
 
 \begin{ttfamily}
 
 \begin{enumerate}
  \item Apply Algorithm \ref{BaseWitt} to $Q_0$ and $y$, and denote by $P$ the result.
  \item Apply Algorithm \ref{Mordell1} to $Q_0'=PQ_0{}^tP$ and denote by $P'$ the result. 
  \item Apply Algorithm \ref{Mordell2} to $P'Q_0'{}^tP'$ and denote by $P''$ the result. 
  \item Return $P=P''P'P$.
 \end{enumerate}
 \end{ttfamily}
 
  \begin{algorithm}\label{Hyperbolique}
 Let $Q_0=(a_{ij})$ of size $n \geq 5$ be defined over $\QQ$ and such that $q_0$ is balanced. This algorithm computes
 $P \in \GLn(\QQ)$ such that $PQ_0{}^tP$ is of the form $\mathbb{H} \oplus \cdots \mathbb{H} \oplus Q_2$ where $Q_2$ is of size $3$ if $n$ is odd, or of size $4$ if $n$ is even.
 \end{algorithm}
 
 \begin{ttfamily}
 
 \begin{enumerate}
  \item Set $P=\Id(n)$ and $i=1$.
  \item Extract the square submatrix $(Q_{0_{jk}})_{i \leqslant j,k \leqslant n}$ and denote it by $Q_2$.
  \item Compute a nonzero rational vector $z$ such that $zQ_0{}^tz=0$.
  \item Apply Algorithm \ref{Redqf} to $Q_2$ and $z$, and denote by $P'$ the result. 
  \item Set $C=\Id(i-1) \oplus P'$. 
  \item Set $P=CP$, $Q_0=CQ_0{}^tC$, $i=i+2$.
  \item If $n-i+1 \geq 5$ go to Step $2$.
  \item Return $P$.
 \end{enumerate}
 \end{ttfamily}
 
 \begin{remark}
  Step $3$ can be done using the algorithm of Castel \cite{PC}, that quickly computes a nonzero rational solution 
 of a rational indefinite quadratic form of dimension $n \geq 5$.
 \end{remark}

 The idea of this algorithm is based on \cite{Mor59}. The main idea is that after each loop the signature changes from $[r,s]$ to $[r-1,s-1]$.
 While the dimension of $Q_2$ is greater or equal to $5$, we can continue because an indefinite quadratic form in $n \geq 5$
 variables has always a nonzero rational solution.

\section{Computation of a nonzero real solution of the system}\label{section:realsolution}

To compute a nonzero real solution of the system of two quadratic forms, we study the nature of the roots of
$\Delta(\lambda)= \det(\lambda Q_0+Q_1)$. Obviously we consider that all the forms in $\mathcal{P}_\RR$ are indefinite, otherwise
$V(\RR)$ is clearly empty (see Theorem \ref{test}). In order to compute a real solution, we are going to first block 
diagonalize the two quadratic forms, and then find a solution using simple algorithms, depending on the number of real roots of $\Delta(\lambda)$.
 
\begin{algorithm}\label{Qfncomp}
 Let $q_0$ and $q_1$ be two quadratic forms of the form \[q_0(x,y,z,w)=x^2-y^2+z^2-w^2,\] \[q_1(x,y,z,w)= ax^2-2bxy-ay^2+cz^2-2dzw-cw^2.\]
 This algorithm computes a nonzero vector $v \in \RR^4$ such that $q_0(v)=q_1(v)=0$.
\end{algorithm}

\begin{ttfamily}

\begin{enumerate}

   \item Set $\varepsilon$ the sign of $b$ and $\varepsilon'$ the sign of $d$. 
   \item Compute a nonzero solution $(x_1,w_1)$ of $|b|x^2-|d|w^2=0$. 
   \item Return $v=(x_1,x_1,-\varepsilon \varepsilon' w_1,w_1)$.

 \end{enumerate}
 \end{ttfamily}

\begin{algorithm}\label{Qfmix}
 Let $q_0$ and $q_1$ be two quadratic forms of the form
 \[q_0(x,y,z)=x^2+y^2-z^2,\] \[q_1(x,y,z)=\lambda_1 x^2+ ay^2-2byz-az^2.\]
 This algorithm computes a nonzero vector $v \in \RR^3$ such that $q_0(v)=q_1(v)=0$.
\end{algorithm}

\begin{ttfamily}

\begin{enumerate}
 \item If $a=\lambda_1$ return $(1,0,1)$.
 \item Let denote by $y_1$,$y_2$ the real solutions of $(a-\lambda_1)y^2 - 2by -(a-\lambda_1)=0$.
 \item If $-y_1^2+1 \geq 0$ return ($\sqrt{-y_1^2+1},y_1,1)$.
 \item Return $(\sqrt{-y_2^2+1},y_2,1)$.

\end{enumerate}

\end{ttfamily}

\begin{proof}
 
 If $a=\lambda_1$, $(1,0,1)$ is an obvious solution of the system.
 Otherwise, we see that the discriminant of \[(a-\lambda_1)y^2 - 2by -(a-\lambda_1)\] is $b^2+(a-\lambda_1)^2>0$   
 hence $y_1$ and $y_2$ are real.
 We have $y_1y_2=-1$ and $y_1+y_2=\frac{2b}{a-\lambda-1}$, therefore $(-y_1^2+1)(-y_2^2+1)= -\frac{4b^2}{(a-\lambda_1)^2}-1<0$. 
 We deduce from this that $-y_1^2+1$ or $-y_2^2+1$ is nonnegative and it is easy to verify that the formula gives a solution. 
 \end{proof}

\begin{lemma}\label{3}
 Let $q_0$ and $q_1$ be two quadratic forms 
 of the form $q_0(x)=\sum_{i=1}^k x_i^2- \sum_{j=k+1}^n  x_j^2$, 
             $q_1(x)= \sum_{i=1}^n b_i x_i^2$ 
 satisfying Hypothesis H over $\RR$.  
 Let denote $m_-=\min(b_i \ | \ i \in \{k+1, \ldots,n \})$ and $m_+=\min(b_i \ | \ i \in \{1, \ldots,k \})$. 
 There exists a real $\lambda$ such that $\lambda q_0 + q_1$ is a positive definite quadratic form
 if and only if $-m_- < m_+$.
\end{lemma}

% \begin{proof}
%  Let assume that there exists a real $\lambda$ such that $\lambda q_0+q_1$ is a definite positive quadratic form. This implies that $\lambda + b_i > 0$ for all $i \leq k$, hence $\lambda + m_+ > 0$. Similarly we have
%  $-\lambda+ m_- > 0$, that implies $-m_- < m_+$. 
%  Conversely, for any $\lambda$ satisfying $-m_- < -\lambda < m_+$, $\lambda q_0 +q_1$ is positive definite.
% \end{proof}

\begin{lemma} \label{5}
  Let $q_0$ and $q_1$ be two quadratic forms 
 of the form $q_0(x)=\sum_{i=1}^k x_i^2- \sum_{k=m+1}^n  x_j^2$, 
             $q_1(x)= \sum_{i=1}^n b_i x_i^2$ 
 satisfying Hypothesis H over $\RR$. 
 Let denote $m_-=\min(b_i \ | \ i \in \{k+1, \ldots,n \})$ and $m_+=\min(b_i \ | \ i \in \{1, \ldots,k \})$. 
 Let denote $M_-=\max(b_i \ | \ i \in \{k+1, \ldots,n \})$ and $M_+=\max(b_i \ | \ i \in \{1, \ldots,k \})$. 
 Then $V(\RR)$ is nonempty if and only if $-m_- \geq m_+$ and $-M_- \leq M_+$
\end{lemma}

% \begin{proof}
%  According to Theorem \ref{test}, $V(\RR)$ is not empty if and only if $\lambda q_0 + q_1$ is never positive definite and never negative definite. The case negative definite is obtained using Lemma \ref{3} by considering
%  $-q_0$ and $-q_1$.
% \end{proof}

\begin{algorithm}\label{Qfnreel}
  Let $q_0$ and $q_1$ be two quadratic forms 
 of the form $q_0(x)=\sum_{i=1}^n a_i x_i^2$ with $a_i= \pm 1$, and
             $q_1(x)= \sum_{i=1}^n b_i x_i^2$ 
 satisfying Hypothesis H over $\RR$, and such that $V(\RR) \neq \emptyset$. 
 This algorithm computes a nonzero vector $w \in \RR^n$ such that $q_0(w)=q_1(w)=0$ .
\end{algorithm}

\begin{ttfamily}

\begin{enumerate}

  \item Let $v_1$ be the list of all the $i$ such that $a_i=+1$ and $v_2$ the list containing the others.
  \item Search $M$ such that $b_M=\max_{i \in v_1}(b_i)$ and $m$ such that $b_m=\min_{i \in v_1}(b_i)$.
  \item If for all $k \in v_2$ we have $-b_k \notin [b_m,b_M]$ then set $q_0=-q_0$ and go to Step $1$.
  \item Choose $k \in v_2$ such that $b_m \leq -b_k \leq b_M$.
  \item If $b_m=-b_k$, set $x_m=1$, $x_k=1$ and $x_i=0$ for the other $i$. 
  \item Otherwise, set $x_M=1$, $x_m=\sqrt{\frac{-b_M-b_k}{b_m+b_k}}$, $x_k=\sqrt{x_m^2+1}$ and $x_i=0$ for the other $i$.
  \item Return $(x_1,\ldots,x_n)$.

\end{enumerate}

\end{ttfamily}

\begin{proof}
For Step 3, Lemma \ref{5} assures that we can find $k$ and $k'$ in $v_2$
 such that $-b_k \geq b_m$ and $-b_{k'} \leq b_M$. If $-b_k \leq b_M$ or $-b_{k'} \geq b_m$, we can go to Step 4. Otherwise we have $-b_{k'} \leq b_m \leq-b_k$. In this case, setting $q_0=-q_0$ 
 exchanges $v_1$ and $v_2$, so that the actual value of $m$ will provide us with a solution for $k$ for the new Step 3.
For Step 5, $(b_M+b_k)$ and $(b_m+b_k)$ have opposite signs and it is an easy exercise to verify that Steps 5 and 6 give us a solution. 
\end{proof}

\begin{algorithm}\label{Qf13reel}
 Let $Q_0$ and $Q_1$ be two matrices of size $n$ satisfying Hypothesis H over $\RR$ and such that $V(\RR) \neq \emptyset$. 
 This algorithm computes a nonzero $y \in \RR^n$ such that $q_0(y)=q_1(y)=0$.
\end{algorithm}

\begin{ttfamily}
 
\begin{enumerate}
 \item Set $\Delta(\lambda)=\det(\lambda Q_0 + Q_1)$.
 \item Set $a$ the number of real roots of $\Delta(\lambda)$.
 \item Apply Algorithm \ref{diagonalisation} to $Q_0$ and $Q_1$. Denote by $P$ the result.
 \item Set $Q_0'=PQ_0{}^tP$ and $Q_1'=PQ_1{}^tP$.
 \item If $a=0$, apply Algorithm \ref{Qfncomp} to $(Q_{0_{ij}}')_{1 \leqslant i,j \leqslant 4}$ and $(Q_{1_{ij}}')_{1 \leqslant i,j \leqslant 4}$. Denote by $y$ the result and
       set $z=(y_1,y_2,y_3,y_4,0,\ldots,0)$. 
 \item If $a=13$ apply Algorithm \ref{Qfnreel} to $Q_0'$ and $Q_1'$. Denote by $z$ the result.
 \item In $0< a < 13$, apply Algorithm \ref{Qfmix} to $(Q_{0_{ij}}')_{a \leqslant i,j \leqslant a+2}$ and $(Q_{1_{ij}}')_{a \leqslant i,j \leqslant a+2}$. Denote by $(z_a,z_{a+1},z_{a+2})$ the result.
       Set $z=(0,\ldots 0, z_a, z_{a+1}, z_{a+2}, 0, \ldots,0)$.
 \item Return $z \cdot P$.      

\end{enumerate}

\end{ttfamily}

\section{A suitable change of basis}\label{section:changebasis}

In this section, we give some algorithms to construct a rational totally isotropic subspace of $q_0(x)$ such 
that $q_1(x)$ is indefinite over this subspace. We keep the notation of the previous section concerning the inputs and outputs of the algorithms.

\begin{algorithm}\label{Mordell3}
 Let $Q_0=(a_{ij})$ be such that $a_{11}=0$ and $a_{13} \neq 0$.  
This algorithm computes a matrix $P \in \GLn(K)$ such that $a'_{11}=0$, $a'_{12}=0$, $a'_{13}=1$, and $a'_{1i}=0$ for $4 \leq i \leq n$. Moreover, the first two columns of $P$ are the same as in $\Id(n)$.
 \end{algorithm}
 
 \begin{ttfamily}
 
 \begin{enumerate}
  \item Set $P=\Id(n)$ and divide the third line of $P$ by $a_{13}$.
  \item Set $Q_0'=PQ_0{}^tP$.
  \item Set $P'=\Id(n)$. Set $P'_{23}=-(Q_0')_{12}$, and for $i=4$ to $n$, $P'_{i3}=-(Q_0')_{i1}$.
  \item Return $P=P'P$.
 \end{enumerate}
 \end{ttfamily}

\begin{algorithm}\label{DoubleWitt}

 Let $Q_0$ and $Q_1$ be two matrices of size $n \geq 3$ satisfying Hypothesis H over $\RR$ and such that $V(\RR) \neq \emptyset$. Let a nonzero $z \in \RR^n$ be such that $q_0(z)=0$ and 
 $q_1(z)=0$. 
 This algorithm computes a matrix $P \in \GLn(\RR)$ such that the first line of $PQ_0{}^tP$ (resp. $PQ_1{}^tP$) is $(0,1,0,\ldots,0)$ (resp. $(0,0,1,0, \ldots,0)$).
\end{algorithm}

\begin{ttfamily}
 
 \begin{enumerate}
  
  \item Apply Algorithm \ref{BaseWitt} to $Q_0$ and $z$, and denote by $P$ the result.
  \item Set $Q_0'= PQ_0{}^tP$ and $Q_1'= PQ_1{}^tP$.
  \item Apply Algorithm \ref{Mordell1} to $Q_0'$ and denote by $P'$ the result. 
  \item Set $P=P'P$ and  $Q_1''= P'Q_1'{}^tP'$. 
  \item Let $i \geq 3$ be the smallest index such that $(Q_1'')_{1i} \neq 0$, and $P''$ be the transposition matrix which exchanges 
        the third line with the $i$-th line. Set $P=P''P$.
  \item Apply algorithm  \ref{Mordell3} to $P''Q_1''{}^tP''$ and denote by $R$ the result. Set $P=RP$.
  \item Return $P$.
  
 \end{enumerate}

\end{ttfamily}

\begin{proof}
 This algorithm is straightforward, until Step $4$. At this step the first line of $PQ_0{}^tP$ is $(0,1,0, \ldots,0)$ and the first line of $Q_1''$ is $(0, *, \ldots,*)$. 
 For Step 5, certainly there exists an index $i \geq 2$ such that $(Q_1'')_{1i} \neq 0$ because $\det(Q_1) \neq 0$. The index $i=2$ cannot be the only one because otherwise $z$ would be a singular point of $V$, 
 contradicting the smoothness of $V$. This proves that Step 5 is always possible. For Step 6, the matrix $R$ does not change the first
 line of $Q_0''$ because the first two columns of $R$ are the same as in $\Id(n)$.
\end{proof}

\begin{algorithm}\label{PosNeg}
 Let $Q_0=(a_{ij})$ and $Q_1=(b_{ij})$ be two matrices of size 
  $n \geq 5$ satisfying Hypothesis H over $\RR$. Let a nonzero $z \in \RR^n$ be such that $q_0(z)=0$ and $q_1(z)=0$.  
  This algorithm computes a nonzero $z^- \in \RR^n$ such that $q_0(z^-)=0$ and $q_1(z^-)<0$.
\end{algorithm}

\begin{ttfamily}
 
 \begin{enumerate}
  
  \item Apply Algorithm \ref{DoubleWitt} to $Q_0$, $Q_1$ and $z$. Denote by $P$ the result and set $Q_0'=PQ_0{}^tP$ and $Q_1'=PQ_1{}^tP$.
  \item Extract the submatrix $(Q_{0_{ij}}')_{2 \leqslant i,j \leqslant n}$ of $Q_0'$, denote it by $F$ and let $f(x)$ be the associated quadratic form.
  \item Extract the submatrix $(Q_{1_{ij}}')_{2 \leqslant i,j \leqslant n}$ of $Q_1'$, denote it by $G$ and let $g(x)$ be the associated quadratic form.
  \item If $F_{22}=0$, set $z=((-1-G_{22})/2,0,1,0,\ldots,0)$ and go to Step $8$. 
  \item Denote by $\varepsilon$ the sign of $F_{22}$. Set $y_1=1$, $y_2=\varepsilon$, and for $i=3$ to $n-1$, set $y_i=0$.
  \item While $g(y)- y_2 f(y) \geq 0$, set $y_2=2 y_2$.
  \item Set $z=(-f(y)/2,y_1,\ldots,y_{n-1})$.
  \item Return $z^-=z\cdot P$.
  
 \end{enumerate}
  
\end{ttfamily}

\begin{proof}
%  Let us assume we have a nonzero rational solution of $q_0(x)=0$, that we denote $y$, such that $q_1(y)>0$. 
%  Then, we search a second rational solution $z$ such that $q_1(z)<0$.
 At the end of Step 3 we have:
 \[ \begin{array}{rcl}
 
 q_0'(x) &=& 2x_1x_2 + f(x_2, \ldots, x_n) \\
 q_1'(x) &=& 2x_1x_3 + g(x_2, \ldots, x_n)
 
 \end{array}\]
For Step 4, we have $q_0'(z)=0$ and $q_1'(z)=-1$.
 Otherwise, we consider the function: 
 \[
  h(x)=x_2g(x)-  x_3f(x).
 \]
This is a polynomial of degree $3$ in $x_3$ and the leading coefficient is  $-F_{22}$. 
  So, $h(x)$ is negative for $x_3= \varepsilon x_3'$ with $x_3'>0$ large enough. Setting at last $y=(1,\varepsilon x_3',0,\ldots,0)$ and $z=(-f(y)/2,1,\varepsilon x_3',0,\ldots,0)$, 
  we have $q'_0(z)=0$ and $q'_1(z)= -\varepsilon x_3' f(y) + g(y) <0$.

\end{proof}

\begin{algorithm}\label{Approx}
 
 Let $Q_0=(a_{ij})$ and $Q_1=(b_{ij})$ be two matrices of size 
  $n \geq 5$, a nonzero $y \in \RR^n$ be such that $q_0(y)=0$ and $q_1(y)<0$. 
 This algorithm computes a $z \in \QQ^n$ such that
  $q_0(z)=0$, $q_1(z)<0$.

\end{algorithm}

\begin{ttfamily}
 
 \begin{enumerate}
 
  \item Compute a rational nonzero solution $w$ of $q_0(w)=0$ and  apply \newline Algorithm \ref{Redqf} over $\QQ$ to $Q_0$ and $w$. 
        Denote by $P'$ the result. Set $Q_0'=P'Q_0{}^tP'$, $Q_1'=P'Q_1{}^tP'$, and $y'=y \cdot P'^{-1}$.
  \item If $y'_i=0$ for all $i \geq 2$, return $(1,0,\ldots,0)P'$.       
  \item Denote $i \geq 2$ the smallest index such that $y_i' \neq 0$. We set $P''$ the permutation matrix that exchanges the 
       second line with the $i$-th. Set $P=P''P'$, $y''=y' \cdot P''$, $Q_0''=P''Q_0'{}^tP''$, $Q_1''=P''Q_1'{}^tP''$, and $\varepsilon = \frac{|y_2''|}{2}$.     
  \item Extract the submatrix $(Q_{0_{ij}}'')_{3 \leqslant i,j \leqslant n}$ of $Q_0''$, denote it by $F$ and let
        $f(x)$ be the associated quadratic form.    
   \item For $i=1$ to $n$, choose a rational number $z''_i$ such that $|y_i''-z_i''|<\varepsilon$. 
   If $q_1''(z'') \geq 0$, set $\varepsilon=\varepsilon / 2$ and go to Step 5. 
  \item Set $u_1=\frac{-f(z_3'',\ldots,z_n'')}{2z_2''}$ and $u=(u_1, z_2'',\ldots,z_n'')$.
  \item If $q_1''(u)<0$, return $u \cdot P$. Otherwise,
        set $\varepsilon= \varepsilon / 2$ and go to Step 5.
 \end{enumerate}

 \end{ttfamily}

 \begin{proof}
 
  After Step $4$, we have $q_0''(y'')=q_0(y)=0$ and $q_1''(y'')=q_1(y)<0$ with $Q_0''= \mathbb{H} \oplus F$. 
  Step 5 is possible because the function $q_1''$ is continuous and the set $\QQ$ is dense in $\RR$. Because $\varepsilon \leq |y_2''|$, we have $z_2'' \neq 0$.
  Since $Q_0''= \mathbb{H} \oplus F$, the formula in Step 6 gives $q_0''(u)=0$. As $y''$ also satisfies the relation $y_1''=\frac{-f(y_3'',\ldots,y_n'')}{2y_2''}$, by continuity we
  deduce that, when $\varepsilon$ is small enough, $u_1$ is close to $y_1''$ so that $u$ is close to $y''$ and $q_1''(u)<0$.
  
 \end{proof}

\section{Computation of a nonzero rational solution}\label{section:rationalsolution}

 \begin{algorithm}\label{Qf13solve}
 Let $Q_0$ and $Q_1$ be two matrices of size $n \geq 13$ satisfying Hypothesis H over $\QQ$ and such that $V(\RR) \neq \emptyset$. 
 This algorithm computes a nonzero $x \in \QQ^n$ such that $q_0(x)=q_1(x)=0$.
\end{algorithm}

\begin{ttfamily}

\begin{enumerate}
 
 \item Apply Algorithm \ref{1} and find $\lambda_0 \in \QQ$ such that $Q_0 + \lambda_0 Q_1$ is balanced and has nonzero determinant. Set $Q_0= Q_0 + \lambda_0 Q_1$.
 \item Compute a nonzero solution $y \in \QQ^n$ of $q_0(y)=0$.
 \item If $q_1(y)=0$ return $y$. If $q_1(y)<0$ set $Q_1=-Q_1$.
 \item Apply Algorithm \ref{Qf13reel} to $Q_0$ and $Q_1$. Denote by $u$ the result.
 \item Apply Algorithm \ref{PosNeg} to ($Q_0$,$Q_1$,$u$) and denote by $v$ the result.
 \item Apply Algorithm \ref{Approx} to ($Q_0$,$Q_1$,$v$) and denote by $z$ the result.
 \item While $q_0(y,z)=0$, do 
            \begin{enumerate}
            \item Choose $y' \in \QQ^n$ randomly until $q_0(y,y') \neq 0$. 
            \item Set $w=y'-\frac{q_0(y')}{2q_0(y,y')}y$.
            \item If $q_1(w)=0$ return $w$.
            \item If $q_1(w)>0$ set $y=w$ otherwise set $z=w$.
            \end{enumerate}
 \item Let $P^{(1)}$ be a matrix whose the first $n-2$ lines generate the \newline solutions in $x$ of $xQ_0{}^ty=xQ_0{}^tz=0$, and whose the last two lines are $y$ and $z$.
 \item Set $Q_0^{(1)}=P^{(1)}Q_0{}^tP^{(1)}$, $Q_1^{(1)}=P^{(1)}Q_1{}^tP^{(1)}$.
 \item Set $P^{(2)}$ the permutation matrix that exchanges the first line with the $(n-1)$-th line and the second line with the $n$-th line.
 \item Set $Q_0^{(2)}=P^{(2)}Q_0^{(1)}{}^tP^{(2)}$, $Q_1^{(2)}=P^{(2)}Q_1^{(1)}{}^tP^{(2)}$ and $P=P^{(2)}P^{(1)}$.      
 \item Extract the submatrix $(Q_{0_{ij}}^{(2)})_{3 \leqslant i,j \leqslant n}$ of $Q_0^{(2)}$, denote it by $Q_2$. 
 \item Apply Algorithm \ref{Hyperbolique} to $Q_2$ and denote by $P'$ the result. 
 \item Set $P^{(3)}=\Id(2) \oplus P'$, $Q_0^{(3)}=P^{(3)}Q_0^{(2)}{}^tP^{(3)}$ , $Q_1^{(3)}=P^{(3)}Q_1^{(2)}{}^tP^{(3)}$ 
       and $P=P^{(3)}P$.
 \item If $(Q_1^{(3)})_{33}>0$, compute a nonzero rational solution of $q_1^{(4)}(x)=0$ of the form $x=(0,x_2,x_3,0,x_5,0,x_7,0,x_9,0,0,0,\ldots)$. Otherwise,
       compute a nonzero rational solution of $q_1^{(4)}(x)=0$ of the form $x=(x_1,0,x_3,0,x_5,0,x_7,0,x_9,0,0,0,\ldots)$.
 \item Return $xP$.

\end{enumerate}

\end{ttfamily}

\begin{theorem}

Let $q_0(x)$ and $q_1(x)$ be two indefinite rational quadratic form in $n \geq 13$ variables satisfying Hypothesis H over $\QQ$ and such that $V(\RR)$ is not empty.  
Then there exists a nonzero rational solution $x$ of $q_0(x)=q_1(x)=0$ . 
Moreover Algorithm \ref{Qf13solve} computes such a solution.
 
\end{theorem}

\begin{proof}

After Step $1$, $q_0(x)$ is balanced, so that in Step 2, such a  rational $y$ exists and after Step 3, we have $q_1(y)>0$. For Step $4$ such a real solution exists because $V(\RR)$ is nonempty. 
Steps 4,5 and 6 compute a rational vector $z$ such that $q_0(z)=0$ and $q_1(z) < 0$.
Step $7$ assures us that $y$ and $z$ are not orthogonal for $q_0$, then the intersection of $\langle y,z \rangle$ and
$\langle y,z \rangle ^{\perp_{q_0}}$ is nonzero. Therefore the matrix $P^{(1)}$ of Step $8$ is invertible. Step $13$ is possible because  
$Q_0^{(2)}= \mathbb{H} \oplus Q_2$ is balanced with signature $[r,s]$ thus $Q_2$ is balanced with signature $[r-1,s-1]$ and dimension
$n-2 \geq 11$. The subspaces of the elements of the form $x=(0,x_2,x_3,0,x_5,0,x_7,0,x_9,0,0,0,\ldots)$ and $x=(x_1,0,x_3,0,x_5,0,x_7,0,x_9,0,0,0,\ldots)$ are both totally isotropic for $q_0^{(4)}$. To conclude we just need to
compute a solution of $q_1^{(4)}(x)=0$ in one of these subspaces. 
Since $(Q_1^{(4)})_{11}>0$ and $(Q_1^{(4)})_{22}<0$, the choice made in Step $15$ assures that $q_1^{(4)}(x)$ is indefinite on this subspace. 
Moreover, in this subspace $q_1^{(4)}(x)$ has $5$ variables and, by the Hasse principle, has rational solutions. This concludes the proof.

\end{proof}

\end{document}